%% file: RR-7555.tex
\documentclass[a4paper,11pt,twoside]{paper}
\usepackage[T1]{fontenc} 
\newif\ifRR\RRfalse
\RRtrue
\ifRR
  \usepackage{RR}
\fi
\usepackage{amsmath,amssymb,amsthm,mathrsfs}
\usepackage{hyperref}
\newtheoremstyle{sans}{\parskip}{\parskip}{\itshape}
                       {0pt}{\bfseries\sffamily}{.}{ }{}
\newtheoremstyle{sansplain}{\parskip}{\parskip}{}
                       {0pt}{\bfseries\sffamily}{.}{ }{}

\theoremstyle{sans}
\newtheorem{prop}{Proposition}[section]

\newtheorem{thm}[prop]{Theorem}
\newtheorem{lem}[prop]{Lemma}

\theoremstyle{sansplain}

\newtheorem{defn}[prop]{Definition}
\newtheorem{exa}[prop]{Example}

\newcommand\bbox{\hfill\rule{2mm}{2mm}}

\makeatletter
\@addtoreset{equation}{section}
\makeatother

\newcommand\C{\mathbb{C}}

\newcommand\Sc{\mathcal{S}}
\newcommand\Zb{\mathbb{Z}}
\newcommand\drift{\overrightarrow{\mathbf{M}}}
\renewcommand{\geq}{\geqslant}
\renewcommand{\leq}{\leqslant}
\def\D{\mathcal{D}}
\def\M{\mathcal{M}}
\def\DD{\displaystyle}

\begin{document}
\ifRR
\RRdate{Mars 2011}
\RRauthor{Guy Fayolle\thanks{INRIA Paris-Rocquencourt, Domaine de Voluceau, BP 105, 78153 Le Chesnay Cedex, France. Email: {\tt Guy.Fayolle@inria.fr}}    \and
        Kilian Raschel\thanks{Fakult\"at f\"ur Mathematik, Universit\"at Bielefeld,
        Postfach 100131, 33501 Bielefeld, Deutschland.
       Email: \texttt{kraschel@math.uni-bielefeld.de}}}

\RRtitle{Marches al\'eatoires dans le quart de plan avec d\'erive nulle:\\ un crit\`ere explicite de finitude pour le groupe associ\'e}
\RRetitle{Random walks in the quarter-plane with zero drift: \\ an explicit criterion for the finiteness\\ of the associated group}
\titlehead{An explicit criterion for the finiteness of the group in the genus 0 case}
\RRresume{Dans plusieurs \'etudes r\'ecentes sur les marches al\'eatoires dans le quart de plan avec des sauts vers les huit plus proches voisins, il apparaît que le comportement de certaines quantités d'int\'er\^et est directement li\'e au \emph{groupe de la marche},  notamment \`a la finitude de son ordre. Pour les marches à d\'erive nulle, nous donnons une formule pour l'ordre de ce groupe, en fonction explicite des probabilités de saut. De fa\c con g\'en\'erale, lorsque le \emph{genre} de la courbe alg\'ebrique d\'efinie par le \emph{noyau} est $0$, le groupe est toujours infini, sauf pr\'ecis\'ement lorsque le saut moyen est nul, auquel cas la finitude est parfaitement possible.} 

\RRabstract{In many recent studies on random walks with small jumps in the quarter plane, it has been noticed that the so-called \emph{group of the walk} governs the behavior of a number of quantities, in particular  through its \emph{order}. In this paper, when the \emph{drift} of the random walk is equal to $0$, we provide an effective criterion giving the order of this group. More generally, we also show that in all  cases where the  \emph{genus} of the algebraic curve defined by the \emph{kernel} is $0$, the group is infinite, except  precisely for the zero drift case, where finiteness is quite possible.}

\RRmotcle{Automorphisme, fonction g\'en\'eratrice, genre, marche al\'eatoire homog\`ene par morceaux, quart de plan, fonction elliptique de Weierstrass.}

\RRkeyword{Automorphism, generating function, genus, piecewise homogeneous random walk, quarter-plane, Weierstrass elliptic function.}

\RRprojet{Imara}

\RRtheme{\THNum}
\RRNo{7555}
\URRocq
\makeRR
\else
\pagestyle{empty}

\title{Random walks in the quarter plane with zero drift: \\ an explicit criterion for the finiteness\\
 of the associated group}

\author{Guy Fayolle\thanks{INRIA Paris-Rocquencourt, Domaine de Voluceau, BP 105, 78153 Le Chesnay Cedex, France. Email: {\tt Guy.Fayolle@inria.fr}}    \and
        Kilian Raschel\thanks{Fakult\"at f\"ur Mathematik, Universit\"at Bielefeld,
        Postfach 100131, 33501 Bielefeld, Deutschland.
       Email: \texttt{kraschel@math.uni-bielefeld.de}}}

\date{\today}
\maketitle

\begin{abstract}
In many recent studies on random walks with small jumps in the quarter plane, it has been noticed that the so-called \emph{group of the walk} governs the behavior of a number of quantities, in particular  through its \emph{order}. In this paper, when the \emph{drift} of the random walk is equal to $0$, we provide an effective criterion giving the order of this group. More generally, we also show that in all  cases where the  \emph{genus} of the algebraic curve defined by the \emph{kernel} is $0$, the group is infinite, except  precisely for the zero drift case, where finiteness is quite possible.
\end{abstract}


\emph{AMS $2000$ Subject Classification: primary 60G50; secondary 30F10, 30D05.}
\newpage
\fi

\section{Introduction and main results}
For several decades, lattice random walks in the Euclidean quarter plane with unit jumps (or steps) have been at the crossroads of several mathematical areas: probability, complex analysis and, more recently, combinatorics. In this context,  one of the basic problems  often amounts to solve functional equations of two complex variables $x,y$, which have the following typical form
  \begin{equation}
     \label{eq_fonc}
          K(x,y)Q(x,y)=k(x,y)Q(x,0)+\widetilde{k}(x,y)Q(0,y)+k_{0}(x,y)Q(0,0)+\kappa(x,y),
     \end{equation}
where 
	\begin{itemize}
\item $K(x,y)$ [usually called the \emph{kernel}], $k(x,y)$, $\widetilde k(x,y)$, $k_0(x,y)$, and
$\kappa(x,y)$ are known functions;
\item $Q(x,y) $ is sought to be analytic in the region \{$|x|, |y| \leq 1$\} and continuous up to the boundary.  
	\end{itemize}
We briefly describle two situations where  equations of type \eqref{eq_fonc} appear.
\medskip
\begin{exa}\label{exa_prob} This example deals with the random walks with jumps of unit length inside the quarter plane, but arbitrary big on the axes. The question is then to calculate the related invariant measure 
$\{\pi_{i,j}\}_{i,j\geq 0}$ as well as to provide conditions for its existence.
Letting 
     \begin{equation*}
          Q(x,y)=\textstyle\sum_{i,j\geq 0}\pi_{i,j}x^{i} y^{j},
     \end{equation*}
the classical Kolmogorov's equations yield (see, e.g., \cite{FIM,FMM}) equation 
\eqref{eq_fonc} with $\kappa(x,y)=0$, whereas $k(x,y)$, $\widetilde k(x,y)$ and $k_0(x,y)$  correspond to the generating functions of the jump probabilities on the horizontal axis, on the vertical axis and at $(0,0)$, respectively.
Letting  $\{p_{i,j}\}_{-1\leq i,j\leq 1}$ denote the jumps in  $\mathbb{Z}_{+}^{2}$, 
the \emph{kernel} is then given by 
     \begin{equation}\label{def_K}
          K(x,y)=\textstyle xy[\sum_{-1\leq i,j\leq 1} p_{i,j}x^i y^j -1].
     \end{equation} 
\end{exa} 

\begin{exa}\label{exa_comb}
The enumeration of planar lattice walks---a classical topic in combinatorics---is the topic of this second example. For 
a given set $\Sc$ of allowed jumps, it is a matter of counting the number of paths of a given length, starting from a fixed origin and ending at an arbitrary point, and possibly restricted to certain regions of the plane. A first basic and natural question arises: how many such paths exist?  
If the paths are confined to $\Zb_{+}^{2}=\{0,1,\ldots \}^{2}$ and if $\Sc$ is included in the set of the eight nearest neighbors, let  $q(i,j;n)$ denote the number of paths of length $n$, which start from $(0,0)$ and end at $(i,j)$. Then  
     \[
          Q(x,y)=\sum_{i,j,n\geq 0}q(i,j;n)x^{i}y^{j}z^{n}
     \]
does satisfy equation \eqref{eq_fonc}, see~\cite{BMM}, where the kernel is now equal to
       \[ 
      K(x,y)=xyz[\textstyle\sum_{(i,j)\in\mathcal{S}}x^i y^j-1/z],
     \]
$k(x,y)=K(x,0)$, $\widetilde k(x,y)=K(0,y)$, $k_0(x,y)=-K(0,0)$ and $\kappa(x,y)=-xy$.
Here $Q(x,y)$ does exist at least in the domain \{$|x|, |y| \leq 1;  |z|\leq 1/|\Sc|$\}, and the third variable $z$ essentially plays the role of a parameter. This situation would also occur in the context of example \ref{exa_prob},  provided we would be interested in the transient behavior of the process. \bbox
\end{exa}
 \medskip
 Throughout this paper, we shall essentially remain in the framework of example \ref{exa_prob}. An exhaustive original method of solution of equation \eqref{eq_fonc} has been given in the book \cite{FIM}, allowing to get explicit  expressions for $Q(x,y)$. It  mainly resorts to a reduction to boundary value problems of Riemann-Hilbert type as in \cite{FI}, together with the analysis of a certain \emph{group} of Galois automorphisms introduced in \cite{MAL} and acting on the algebraic curve (and its related Riemann surface) 
  	\begin{equation}
     \label{def_Riemann}
          \mathscr{K}=\{(x,y)\in\mathbb{C}^{2} : K(x,y) = 0\}.     
     \end{equation}

Let $\C\,(x,y)$ be the field of rational functions in $x,y$ over $\C$. Since we shall assume that $K$, given by \eqref{def_K}, is irreducible (see \cite[Lemma 2.3.2]{FIM} for an interpretation of this hypothesis in terms of the parameters $\{p_{i,j}\}_{-1\leq i,j\leq 1}$), the quotient field $\C(x,y)$ denoted by $\C_Q(x,y)$ is also a field.
 \
  \begin{defn} \label{def_group}
	The \emph{group of the walk} is the Galois group  $W=\langle \xi,\eta\rangle$ of automorphisms of   $\C_Q(x,y)$ generated by  $\xi$ and $\eta$ given by
	 \begin{equation*}
     \label{def_generators_group}
          \xi(x,y)= \Bigg(x,\frac{1}{y}\frac{\sum_{-1\leq i \leq 1}p_{i,-1}x^{i}}
          {\sum_{-1\leq i \leq 1}p_{i,+1}x^{i}}\Bigg),
          \ \ \ \ \ \eta(x,y)=\Bigg(\frac{1}{x}\frac{\sum_{-1\leq j\leq 1}p_{-1,j}y^{j}}
          {\sum_{-1\leq j\leq 1}p_{+1,j}y^{j}},y\Bigg).
     \end{equation*}
\end{defn}
 
 Let
\begin{equation}\label{eq_delta}
\delta = \eta \circ \xi .
\end{equation}
Then $W$ has a normal cyclic subgroup $W_0=\{\delta^\ell, \ell\in\Zb\,\}$, which is possibly infinite, and such that $W/W_0$ is a group of order $2$. Thus, when  $W$ is finite, say of order $2n$, $\delta^{n}$ is the identity.  
   
In many studies, stemming either from probability theory \cite{FIM,KuRa,KRSp4} and, more recently, from combinatorics \cite{BK,BMM,FR,MM2}, the analysis of type \eqref{eq_fonc} equations is often crucial, and the importance of the group $W$ has repeatedly been noticed. Broadly speaking, it appears that the question of the finiteness of $W$ can be very useful in the three following respects.
     \begin{itemize}
          \item To find an explicit form of the generating function $Q(x,y)$, see \cite{BK,BMM,FIM}. 
          \item To prove the holonomy---or even the algebraicity---of $Q(x,y)$, see \cite{BK,BMM,FIM,FR}.
          \item To derive asymptotics of the Taylor coefficients of $Q(x,y)$, see \cite{KuRa,KRSp4}.
     \end{itemize}

\subsection{Brief overview} \label{sec_art}
In spite of the facts recalled above, very few criteria exist in the literature to decide about the finiteness of the group $W$. Up the knowledge of the authors, the main known results can be found in \cite[Chapters 3 and 4]{FIM}. For instance, it is proved that the group $W$ has order $4$ if and only if
     \begin{equation}
     \label{def_Delta}
          \Delta = \left|\begin{array}{lll}
          p_{1,1}&p_{1,0}&p_{1,-1}\\
          p_{0,1}&p_{0,0}-1&p_{0,-1}\\
          p_{-1,1}&p_{-1,0}&p_{-1,-1}
          \end{array}\right|=0.
     \end{equation}
For groups of order $6$, a criterion, still in terms of certain determinants, is also given in \cite[Section 4.1.1]{FIM}. 
  
It has been shown in  \cite[Lemma 2.3.10]{FIM} that for all non-singular random walks (see Definition \ref{def_singular}) the Riemann surface associated with  \eqref{def_Riemann} has always  \emph{genus}~$1$,  except in five cases, listed in  Lemma \ref{lem_genus0} of the appendix, where it has \emph{genus}~$0$.

In the case of genus $1$, a general  criterion has been derived in \cite[Section 4.1.2]{FIM} for $W$ to be of order $2n$, $4\leq n\leq \infty$. This characterization says that the group $W$ is finite if, and only if, the ratio of two elliptic integrals $\omega_3/\omega_2$ 
(see \eqref{eq_omega123} for more details) is \emph{rational}, in which case, setting $
\Zb_{+}^*=\{1,2,\ldots \}$, the order of the group is equal to 
\begin{equation}\label{eq_order1}
2\inf\{\ell\in\mathbb{Z}_{+}^*: \ell \omega_3/\omega_2 \in \mathbb{Z}\}.
\end{equation}
 On other hand, this nice theoretical formula is not easy to rewrite concretely in terms of the parameters $\{p_{i,j}\}_{-1\leq i,j\leq 1}$. For instance, an expression (if any!) by means of determinants, as in 
 \eqref{def_Delta}, is absolutely unclear.

As for the genus $0$ case, although  the solution of equation \eqref{eq_fonc} was constructed  in \cite[Chapter 6]{FIM}, no criterion was yet provided to decide about the finiteness of the group: this is precisely  the subject of this paper, where we give an explicit criterion for $W$ to be finite. The method relies mainly on a \emph{continuity} argument from the genus $1$ case, in a sense to be made precise in Sections \ref{partIT} and \ref{partIIT}.

\subsection{A criterion for the finiteness of W in the genus 0 case}
Now we state our main result, according to the classification of Lemma \ref{lem_genus0} listed in the appendix. As the reader will realize, most of the analysis will be devoted to walks having a zero drift, i.e., $\drift=0$, since it appears to be the most difficult (and maybe interesting!)  situation. Introduce the correlation coefficient
	\[
	R= \frac{\sum_{-1\leq i,j\leq 1}i j p_{i,j}}
          {[\sum_{-1\leq i,j\leq 1}i^2 p_{i,j}]^{1/2}\cdot [\sum_{-1\leq i,j\leq 1}j^2 p_{i,j}]^{1/2}},
          \]
 and define the angle
     \begin{equation} \label{exp_Chi}
          \theta=\arccos (-R).
     \end{equation}
     
\begin{thm} \label{main_thm} \mbox{ }
\begin{itemize}
\item[\emph{(I)}] When $\drift=0$, the group $W$ is finite if, and only if, $\theta/\pi$ is rational, in which case the order of $W$ is equal to 
\begin{equation}\label{eq_order0}
2\inf\{\ell\in\Zb_{+}^*: \ell \theta/\pi \in \Zb\}.
\end{equation}
\item[\emph{(II)}] When $\drift\ne0$, the order of W is always infinite in the four cases  
\eqref{eq_cas2}, \eqref{eq_cas3}, \eqref{eq_cas4} and \eqref{eq_cas5}.\bbox
\end{itemize} 
\end{thm} 

In Section \ref{sec_criter2}, we shall propose another theoretical form of the angle \eqref{exp_Chi}, merely involving characteristic values  related to the uniformization of curves of genus $0$.   

\section{Proof of Part (I) of Theorem \ref{main_thm}}
\label{partIT}

The proof proceeds in stages, the key idea being to consider the genus $0$ case as a \emph{continuous limit} of the genus $1$ case. In particular, we shall consider ad hoc \emph{limit conformal gluing} and \emph{limit uniformizing} functions, allowing to connect the ratio of certain \emph{limit periods} with the finiteness of the \emph{limit group}.

\subsection{Basic properties of the kernel} \label{properties_kernel}
To render the paper as self-contained as possible, we recall hereafter some important results proved in \cite{FIM}. They are needed for our purpose and they will be stated without further comment. 

First, let us rewrite the kernel \eqref{def_K} in the two equivalent forms
     \begin{equation*}
     \label{eq_kernel}
          K(x,y) = a(x) y^{2}+ b(x) y + c(x) = \widetilde{a}(y) x^{2}+
          \widetilde{b}(y) x + \widetilde{c}(y),
     \end{equation*}
where
     \begin{equation*}
     \label{def_a_b_c}
          \begin{array}{llllllllll}
               \! a(x) =& \hspace{-3mm}p_{1,1}x^{2}+& \hspace{-3mm}p_{0,1}x+& \hspace{-3mm}p_{-1,1},
               \ b(x) =& \hspace{-3mm}p_{1,0}x^{2}+& \hspace{-3mm}(p_{0,0}-1)x+& \hspace{-3mm}p_{-1,0},
               \ c(x) =& \hspace{-3mm}p_{1,-1}x^{2}+& \hspace{-3mm}p_{0,-1}x+& \hspace{-3mm}p_{-1,-1},\\
               \! \widetilde{a}(y) =& \hspace{-3mm} p_{1,1}y^{2}+& \hspace{-3mm}p_{1,0}y+& \hspace{-3mm}p_{1,-1}, 
               \ \widetilde{b}(y) =& \hspace{-3mm}p_{0,1}y^2+& \hspace{-3mm}(p_{0,0}-1)y+& \hspace{-3mm}p_{0,-1}, 
               \ \widetilde{c}(y) =& \hspace{-3mm}p_{-1,1}y^{2}+& \hspace{-3mm}p_{-1,0}y+& \hspace{-3mm}p_{-1,-1}.
               \end{array}
     \end{equation*}
Then we set
     \begin{equation*}
     \label{def_d}
          D(x) = b^{2}(x)-4a(x) c(x), 
          \qquad \widetilde{D}(y)=
          \widetilde{b}^{2}(y)-4 \widetilde{a}(y)\widetilde{c}(y).
     \end{equation*}     
The polynomials $D$ and $\widetilde{D}$  are of degree $4$, with respective dominant coefficients
     \begin{equation} \label{eq_C}
          C=p_{1,0}^{2}-4p_{1,1}p_{1,-1},
          \qquad \widetilde{C}=p_{0,1}^{2}-4p_{1,1}p_{-1,1}.
     \end{equation}

Let $X(y)$ [resp.\ $Y(x)$] be the algebraic function defined by \eqref{def_K}, so that $K(X(y),y)=0$ [resp.\ $K(x,Y(x))=0$]. This function has two \emph{branches}, say $X_0$ and $X_1$ [resp.\ $Y_0$ and $Y_1$], and we fix the notation by taking $|X_0|\leq |X_1|$ [resp.\ $|Y_0|\leq |Y_1|$]. 

Denote by $\{y_\ell\}_{1\leq \ell\leq 4}$ the four roots of $\widetilde{D}(y)$---they are branch points of the Riemann surface \eqref{def_Riemann}. They are enumerated in such a way that $|y_1|\leq|y_2|\leq|y_3|\leq|y_4|$. Moreover $y_1\leq y_2$, $[y_1y_2] \subset [-1,+1]$ and $0\leq y_2\leq y_3$. The branches $X_0$ and $X_1$ are meromorphic on the complex plane cut along $[y_1y_2] \cup [y_3y_4]$. Similar results hold for $D(x)$, exchanging $y$ and $x$.
 
On the cut $[y_1y_2] \cup [y_3y_4]$,  the roots $X_0$ and $X_1$ are complex conjugate. Then we consider the two components of the quartic curve
\begin{equation*}
\begin{cases} 
\M_1 &= \
X_0([\underleftarrow{\overrightarrow{y_1y_2}}])\ \, =\ \, \overline{X}_1([\underrightarrow{\overleftarrow{y_1y_2}}]), \\[0.2cm] 
\M_2 &= \ X_0([\underleftarrow{\overrightarrow{y_3y_4}}])\ \, =\ \, \overline{X}_1
([\underrightarrow{\overleftarrow{y_3y_4}}]), \end{cases}
\end{equation*}
 where  $\underleftarrow{\overrightarrow{y_1y_2}}$ stands for the
contour $[y_1 y_2]$, traversed from $y_1$ to $y_2$ along the
upper edge of the slit $[y_1 y_2]$, and then back to $y_1$ along the
lower edge of the slit. Similarly, $\underrightarrow{\overleftarrow {y_1y_2}}$
is defined by exchanging ``upper'' and ``lower''. In particular, the simply connected domain $\D_1$ bounded by $\M_1$ will play here a fundamental role, for reasons explained in Section \ref {glu1}.

\subsubsection{Uniformization in the genus 1 case} \label{unif1}
When the associated Riemann surface is of genus $1$, the algebraic curve $K(x,y)=0$ admits a uniformization in terms of the Weierstrass $\wp$ function with periods $\omega_1,\omega_2$ (see equation \eqref{eq_omega123}) and its derivatives. Indeed, setting   
\begin{eqnarray*}
D(x) & = & b^2(x) -4a(x)c(x)\ \; = \ \; d_{4}x^{4}+d_{3}x^{3}+d_{2}x^{2}+d_{1}x+d_{0},\\
z & = & 2a(x)y+b(x),
\end{eqnarray*}
the following formulae hold [noting that for notational convenience $d_4$ corresponds to $C$ defined in equation \eqref{eq_C}]. 
 \begin{itemize}
\item If $d_4 \neq 0$ (four finite branch points $\{x_\ell\}_{1\leq \ell\leq 4}$), then $D'(x_4)>0$ and
\begin{equation} \label{eq_unif1}
\begin{cases}
x(\omega) = x_4 + \dfrac{D'(x_4)}{\wp(\omega) - D''(x_4)/6}, \\[3ex] 
z(\omega) = \dfrac{D'(x_4) \wp'(\omega)}{2[\wp(\omega) -D''(x_4)/6]^2}. 
\end{cases}
\end{equation}
\item If $d_4 = 0$ (three finite branch points $\{x_\ell\}_{1\leq \ell\leq 3}$ and $x_4=\infty$), then
\begin{equation} \label{eq_unif2}
\begin{cases}
x(\omega) = \dfrac{\wp(\omega) - {d_2}/{3}}{d_3}, \\[2.5ex]
z(\omega) = - \dfrac{\wp'(\omega)}{2d_3}. 
\end{cases}
\end{equation}
\end{itemize}
Let us now introduce the three key quantities
   \begin{equation} \label{eq_omega123}
          \omega_1 = 2i \int_{x_1}^{x_2}\frac{\text{d}x}{\sqrt{-D(x)}},
          \qquad \omega_2 = 2\int_{x_2}^{x_3}\frac{\text{d}x}{\sqrt{D(x)}},
          \qquad \omega_3 = 2\int_{X(y_1)}^{x_1}\frac{\text{d}x}{\sqrt{D(x)}}.
     \end{equation}
The period $\omega_1$ is clearly purely imaginary, and $0<\omega_3<\omega_2$, see \cite[Lemma 3.3.3]{FIM}. In this respect, for $\ell\in\{2,3\}$, it will be convenient to denote by $\wp_{1,\ell}$ the Weierstrass elliptic function with periods 
$\omega_1, \omega_\ell$ and series expansion
     \begin{equation}
     \label{expansion_wp}
          \wp_{1,\ell}(\omega)=\frac{1}{\omega^2}+\sum_{(p_1,p_\ell)\in\mathbb{Z}^{2}\setminus \{(0,0)\}}
          \left[\frac{1}{(\omega-p_1\omega_1-p_\ell\omega_\ell)^{2}}-
          \frac{1}{(p_1\omega_1+p_\ell\omega_\ell)^{2}}\right].
     \end{equation}
On the universal covering,  we have $\delta(\omega+\omega_3)=\delta(\omega)$, see \eqref{eq_delta} and \cite[equation (3.1.10)]{FIM}, and the necessary and sufficient condition for the group to be finite has been quoted in \eqref{eq_order1}. 

\subsubsection{Conformal gluing in the genus 1 case} \label{glu1}
\begin{defn}
     Let $\D\subset\C\cup \{\infty\}$ be an open and simply connected set,
     symmetrical with respect to the real axis, and different from $\emptyset$, $\C$
     and $\C\cup \{\infty\}$. A function $w$ is said to be a \emph{conformal gluing
     function} (CGF) for the set $\D$ if
     \begin{itemize}
     \item $w$ is meromorphic in $\D$;
     \item $w$ establishes a conformal
     mapping of $\D$ onto the complex plane cut along a segment;
     \item For all $t$ in the boundary of $\D$, $w(t)=w(\overline{t})$.
     \end{itemize}
\end{defn}

For instance, $w(t)=t+1/t$ is a CGF for the unit disc, and any non-degenerate linear transformation of $w(t)$, namely 
\[
\frac{ew(t)+f}{gw(t)+h},\qquad (e,g)\ne (0,0), \qquad eh-fg\ne0
\] 
is also a CGF for the unit disc. Incidentally, this implies one can choose arbitrarily the pole of $w$ within the unit disc---in particular, taking $e=0$, $f=1$, $g=1$ and $h=-2$, we get the CGF $t/(t-1)^2$ with a pole at $1$. Conversely, two CGFs for a same domain are fractional linear transformations of each other, see \cite{LIT}.   

\subsubsection{Some topological facts in the genus 0 case} \label{genus0}
As $\drift=0$, the branch points $x_2$ and $x_3$ coincide and  we have 
\[
x_2=x_3=1,  \quad x_1 \in[-1,1), \quad x_4 \in (1,\infty) \cup \{\infty\} \cup (-\infty,-1], 
\]
Similar results hold for the $\{y_\ell\}_{1\leq \ell \leq 4}$.  

Here the curve $\M_1$ and $\M_2$ intersect at the point $x=1$, which is a corner point. 
In the figure \ref{FIG}, borrowed from \cite[Chapter 6]{FIM}, the dotted curve is the unit circle, and one
has drawn the contour $\M_1\cup\M_2$, which has a self-intersection
and is the image of the cut
$[\underleftarrow{\overrightarrow{y_1y_4}}]$ by the mapping
$y\mapsto X(y)$, remembering that
$X_0([\underleftarrow{\overrightarrow{y_1y_4}}])=
\overline{X}_1([\underrightarrow{\overleftarrow {y_1y_4}}])$.
\begin{figure}[htb]
\vspace{1cm}
\begin{center} \input curve.pstex_t \end{center}
\caption{The contour $\M_1\cup\M_2$ for $R<0$}\label{FIG}
\end{figure}

In the sequel, we shall prove that the angle of the tangent line at the corner point is indeed deeply related to the order of the group.

\subsection{Limit conformal gluing when passing  from genus 1 to genus 0} \label{glulimit}
Our approach resides in viewing  genus 0 as a topological deformation of genus $1$, when the parameters are modified in such a way that $\drift\to0$.

According to the uniformization \eqref{eq_unif1} or \eqref{eq_unif2}, and together with the notations of Section \ref{unif1}, let 
     \begin{equation*} 
          f(t)=\left\{\begin{array}{lll}
           \DD \frac{D''(x_4)}{6}+ \frac{D'(x_4)}{t-x_4}& \text{if} & x_4 \neq \infty,\\[0.2cm]
          \DD \frac{1}{6}(D''(0)+ D'''(0) t)& \text{if} & x_4 = \infty.
          \end{array}\right.
     \end{equation*}

To avoid technicalities of minor importance, we shall assume in the sequel $C\ne0$, with $C$ given by \eqref{eq_C}, so that  $x_4\ne\infty$. Then, in the genus $1$ case, it was shown in \cite{FIM} that  a CGF for the domain $\D_1$ bounded by $\M_1$  is obtained via the function
     \begin{equation*}
          w(t)=\wp_{1,3}(\wp_{1,2}^{-1}(f(t))-[\omega_{1}+\omega_{2}]/2).
     \end{equation*}
Since any fractional linear transformation of a CGF is again a CGF, the well-known addition theorem for 
$\wp_{1,3}$ involving translation by the half-period $\omega_1/2$ entails that 
     \begin{equation} \label{after_hp}
          \wp_{1,3}(\wp_{1,2}^{-1}(f(t))-\omega_{2}/2)
     \end{equation}
is also a CGF for $\D_1$.

Now letting $\drift\to0$, so that $x_2, x_3\to1$, and by using the continuity with respect to the parameters $\{p_{i,j}\}_{-1\leq i,j\leq 1}$, a direct calculation in the integral formulae \eqref{eq_omega123} gives that  $\omega_1\to i\infty$ and  that
 $\omega_2, \omega_3$ converge to certain non-degenerate quantities as below:

     \begin{equation} \label{eq_omega0}
   \begin{cases}
     \omega_1\to i\infty , \\[0.2cm]
     \DD\omega_2\to \alpha_2=\frac{\pi}{[C(x_4-1)(1-x_1)]^{1/2}}, \\[0.4cm]
    \DD \omega_3\to \alpha_3=\int_{X_0(y_1)}^{x_1} \frac{\text{d}x}{(1-x)[C(x-x_1)(x-x_4)]^{1/2}}.
    \end{cases}
     \end{equation}

\begin{lem}\label{teta1}
Letting $\drift\to0$, we have
\[
\frac{\theta}{\pi}\ =\ \lim_{\drift\to0} \ \frac{\omega_2}{\omega_3}  \ =\  \frac{\alpha_2}{\alpha_3}. 
\]
\end{lem}
 \begin{proof}
 At the corner point of the curve $\M_1$, we shall calculate the angle of the tangent line with the horizontal axis  in two different ways. 

First,  when $\drift\to0$, the derivative $X_0'(1)$ is a root of a second degree
polynomial, since
\[
\widetilde{a}(1){X_0'}^{2}(1) +RX_0'(1) + a(1) =0,
\]
see \cite[equation (6.5.3)]{FIM}. Since 
\[\textstyle
2a(1)=\sum_{-1\leq i,j\leq 1}j^2 p_{i,j}, \quad 2\widetilde a(1)=\sum_{-1\leq i,j\leq 1}i^2 p_{i,j}, 
\]
it is quite easy to check that
     \begin{equation}
     \label{arg_X_0}
          \arg(X_0'(1))=\pm\theta,
     \end{equation} 
where $\theta$ is given by equation \eqref{exp_Chi}.
  
The second way of computing $\theta$ relies on the construction of  a convenient CGF, according to the definition given in Section \ref{glu1}. Here we shall choose  $w$ with $w(1)=\infty$. As we shall see, this implies that, in the neighborhood of $t=1$, $w(t)$ has the form
     \begin{equation}\label{eq_chi}
          w(t)=[\alpha+o(1)]/[1-t]^{\chi},
     \end{equation}
for some constant  $\alpha\neq 0$ and $\chi>0$. In addition, the exponent $\chi$ must satisfy the relation
     \begin{equation}
     \label{link_chi_theta}
          \chi=\pi/\theta.
     \end{equation}
Indeed, equation \eqref{arg_X_0} together with the symmetry of the CGF yield  $\exp(i \theta \chi)=\exp(-i \theta \chi)$, whence $\theta\chi/\pi$ is a positive integer. On the other hand, if $\theta\chi/\pi\geq 2$, then $w$ would not be one-to-one. Identity \eqref{link_chi_theta} follows.
Thus we are left  with the proof of expansion \eqref{eq_chi}.
%

In order to obtain a CGF in the zero drift case, we could use \cite[Lemma 6.5.5]{FIM}, but it will be more convenient to use here the CGF obtained from the last paragraph, and then to let the drift go to zero. \\ 
When $\drift\to0$, we know from \eqref{eq_omega123} that $\omega_1\to i\infty$. This way, using  \eqref{expansion_wp} as well as the well-known identity
     \begin{equation*}
          \sum_{p\in\mathbb{Z}} \frac{1}{(\Omega+p)^2}=\frac{\pi^2}{\sin^2(\pi\Omega)},
     \end{equation*}
we obtain, for $\ell\in\{2,3\}$, that uniformly in $\omega$,
     \begin{equation*}
          \wp_{1,\ell}(\omega)\to \left(\frac{\pi}{\alpha_\ell}\right)^{2}\left[\frac{1}{\sin^2(\pi \omega/\alpha_\ell)}-\frac{1}{3}\right].
     \end{equation*}
In particular, setting
     \begin{equation*}
          u(t)=\sin^2\left(\frac{\alpha_2}{\alpha_3}
          \left[\arcsin\left\{\left[\frac{1}{3}+f(t)\left(\frac{\omega_2}{\pi}\right)^2\right]^{-1/2}\right\}-\frac{\pi}{2}\right]\right),
     \end{equation*}
and taking the limit in \eqref{after_hp}, it follows that  an admissible CGF for $\M_1$ is given by 
     \begin{equation*}
          \left(\frac{\pi}{\alpha_3}\right)^{2}\left[\frac{1}{u(t)}-\frac{1}{3}\right].
     \end{equation*}
Since any linear transformation of a CGF is a CGF, $u(t)$ itself is a CGF for $\M_1$. We shall now show the existence of $\alpha\neq 0$ such that in the neighborhood of $t=1$,
     \begin{equation}
     \label{behavior_near_1}
          u(t)=[\alpha+o(1)]/[1-t]^{\alpha_2/\alpha_3}.
     \end{equation}

By a direct calculation, it can be seen that for $t\in[x_1,1]$, 
     \begin{equation}
     \label{inner_function}
          \frac{1}{3}+f(t)\left(\frac{\alpha_2}{\pi}\right)^2
     \end{equation}
belongs to the segment $[0,1]$, equals $1$ at $x_1$ and $0$ at $1$. In other words, in order to understand the behavior of $u(t)$ near $t=1$, it is necessary to analyze the asymptotics, as $T\to\infty$, of the function
     \begin{equation*}
          \sin^2 \left(\frac{\alpha_2}{\alpha_3}\left[\arcsin\{T\}-\frac{\pi}{2}\right]\right).
     \end{equation*}
For $T\geq 1$, we have
     \begin{equation*}
          \arcsin\{T\}=\int_{0}^{1}\frac{\text{d}u}{(1-u^2)^{1/2}}\pm i \int_{1}^{T}\frac{\text{d}u}{(u^2-1)^{1/2}}=\frac{\pi}{2}\pm i \ln \bigl[T+(T^2-1)^{1/2}\bigr].
     \end{equation*}
Hence, with $\sin( i x)=i\sinh (x)$, we can write
     \begin{align*}
          \sin^2\left(\frac{\alpha_2}{\alpha_3}\left[\arcsin\left\{T\right\}-\frac{\pi}{2}\right]\right)&=
          -\sinh^2\left(\frac{\alpha_2}{\alpha_3}\ln \bigl[T+(T^2-1)^{1/2}\bigr] \right)\\
          &=-\frac{1}{4} \left(\bigl[T+(T^2-1)^{1/2}\bigr]^{2\alpha_2/\alpha_3}+\bigl[T-(T^2-1)^{1/2}\bigr]^{2\alpha_2/\alpha_3}-2\right).
     \end{align*}     
When $T\to\infty$, $T+(T^2-1)^{1/2}=2T+O(1/T)$ and $T-(T^2-1)^{1/2}=O(1/T)$, so that
     \begin{equation*}
          u(t)=-\frac{1}{4}\Bigg(\left[2\left[\frac{1}{3}+f(t)\left(\frac{\alpha_2}{\pi}\right)^2\right]^{-1/2}\right]^{2\alpha_2/\alpha_3}+O(1)\Bigg).
     \end{equation*}
As remarked earlier, the function defined in \eqref{inner_function} has a zero at $1$, which is simple. Equation \eqref{behavior_near_1} follows immediately, and, since conformal mapping conserves angles, the proof of the lemma is concluded.
\end{proof}

\subsection{Limit of the uniformization when passing  from genus 1 to genus 0}
\label{limit_uni_10}
The last step  consists in connecting the angle $\theta$ to the group $W$ introduced in Definition \ref{def_group}. In fact, we shall deal with the ratio $\alpha_3/\alpha_2$ and prove that, when $\drift=0$, $W$ can be interpreted as the group of transformations of $\mathbb{C}/(\alpha_2\mathbb{Z})$
     \begin{equation*}
          \langle \omega \mapsto-\omega+\alpha_2,\omega \mapsto -\omega+\alpha_2+\alpha_3\rangle,
     \end{equation*}
which is of order $2\inf\{\ell\in\mathbb{Z}_{+}^*: \ell\alpha_3/\alpha_2\in\mathbb{Z}\}$. 

As $\drift\to0$, the limit of the uniformization \eqref{eq_unif1}, after some algebra using the continuity with respect to the parameters $\{p_{i,j}\}_{-1\leq i,j\leq 1}$, becomes 
\begin{equation}\label{eq_limunif1}
\begin{cases}
\DD x(\omega) =1 + \frac{(x_4-1)(1-x_1)}{1- x_1- (x_4- x_1)\sin^2({\pi\omega}/{\alpha_2})}, \\[0.5cm]
\DD z(\omega)= \frac{\pi}{\alpha_2}\frac{(x_4-x_1)(x_4-1)(x_1-1)\sin({2\pi\omega}/{\alpha_2})}
{[1- x_1- (x_4- x_1)\sin^2 ({\pi\omega}/{\alpha_2})]^2} ,
\end{cases}
\end{equation}
where we have used the expression of $\alpha_2$ given by equation \eqref{eq_omega0}.
Setting $u= \exp({2i\pi\omega}/{\alpha_2})$,
 the obvious identity
\[
\sin^2 ({\pi\omega}/{\alpha_2}) = -(u +1/u -2)/4
\]
allows to rewrite in \eqref{eq_limunif1} as $x(\omega)= \widetilde{x}(u)$, where
\begin{equation}\label{eq_limunifx}
\widetilde{x}(u) = \frac{(u-z_1)(u-1/z_1)}{(u-z_0)(u-1/z_0)},
\end{equation}
and $z_0, z_1$ are complex numbers given by
\begin{equation*}
\begin{cases}
\DD z_0 = \frac{2 -(x_1+x_4) \pm 2[(1-x_1)(1-x_4)]^{1/2}}{x_4-x_1}, \\[0.3cm]
\DD z_1 = \frac{x_1+x_4 -2x_1x_4 \pm 2[x_1x_4(1-x_1)(1-x_4)]^{1/2}}{x_4-x_1} .
\end{cases}
\end{equation*}
 
Clearly, $\widetilde{y}(u)$ will be also a rational function of the same form as $\widetilde{x}(u)$, but its explicit computation hinges on some general properties of  Riemann surfaces. Indeed the functions \eqref{eq_limunifx} and $\widetilde{y}(u)$ are \emph{automorphic} and provide a rational uniformization of the algebraic curve  $\mathscr{K}$ defined in \eqref{def_Riemann}, which is  here of genus $0$. Moreover, this unifomization is unique \emph{up to a fractional linear transformation} (see, e.g., \cite{SPR}). 
 
 Before deriving the formula for 
$\widetilde{y}(u)$, let us quote right away a first limit automorphism, namely $\omega\mapsto -\omega +\alpha_2$, since
\[x(\omega)=x(-\omega +\alpha_2).
\] 
Equivalently, we have
\[
\widetilde{x}(u) =\widetilde{x}(1/u),
\]
which corresponds precisely to the automorphism of $\mathbb{C}$
\begin{equation}\label{eq_autx}
 \xi(u)=\frac{1}{u}.
 \end{equation}   

Similarly, exchanging the roles of $x$ and $y$, a continuity argument when $\drift\to0$ in the general uniformization given in  \eqref{eq_unif1} or \eqref{eq_unif2} yields
 \[
 y(\omega) = y(-\omega + \alpha_2 +\alpha_3).
\]
In particular, $\omega\mapsto -\omega + \alpha_2 +\alpha_3$ is the second automorphism. Indeed, $z(\omega)=-z(-\omega)$, so that $y(\omega) \ne y(-\omega)$. Moreover
$x(\omega)\ne x(-\omega +\alpha_3)$, since, by taking the limit of the genus $1$ in 
 \eqref{eq_omega123}, we get $\alpha_3<\alpha_2$ with a strict inequality (see formula \eqref{eq_criter2} in Section \ref{sec_criter2}). 
 
 To derive $\widetilde{y}(u)$ quickly (and nicely!), we shall make two observations.
 
\begin{itemize}
\item Exchanging the roles of $x$ and $y$ would lead to another rational uniformization with some parameter $v$, such that
\begin{equation}\label{eq_limunif2}
\widehat{y}(v) =  \frac{(v-z_3)(v-1/z_3)}{(v-z_2)(v-1/z_2)},
\end{equation}
where $z_2, z_3$ are obtained from $z_0,z_1$, just replacing  $x_1,x_4$ by $y_1,y_4$, respectively.
\item From the above remark, we necessarily have 
\[
\widetilde{y}(u) = \widehat{y}(\sigma(u)),
\]
where
\[
\sigma(u) = \frac{eu+f}{gu+h}, \qquad (e,g)\ne (0,0), \qquad eh-fg\ne0
\]
is a linear transformation which will be completely determined.
\end{itemize}
Since $\widetilde{x}(0)= \widetilde{x}(\infty)=1$, see \eqref{eq_limunifx},
we must have $\widetilde{y}(0)=\widetilde{y}(\infty)=1$ (in the case $\drift =0$, $Y_0(1)=Y_1(1)=1$, see Section \ref{properties_kernel}). Similarly,  \eqref{eq_limunif2} yields $\widehat{y}(0)=\widehat{y}(\infty)=1$. These simple equalities entail at once $\sigma(u)=\rho u$, where $\rho$ is a complex number, to be calculated later. Then, by \eqref{eq_limunif2}, we obtain
\begin{equation}\label{eq_limunif3}   
 \widetilde{y}(u) =  \frac{(\rho u-z_3)(\rho u-1/z_3)}{(\rho u-z_2)(\rho u-1/z_2)},
 \end{equation}
and the second limit automorphism $\eta$  becomes on $\mathbb{C}$
\begin{equation}\label{eq_auty}
 \eta(u) =  \frac{1}{\rho^2u}.
 \end{equation}

Hence, combining \eqref{eq_autx} and  \eqref{eq_auty}, the generator $\delta$ introduced in \eqref{eq_delta} takes on $\mathbb{C}$ the simple form
\[
\delta(u) = (\eta \circ \xi) (u) = \frac{u}{\rho^2}.
\]
In the genus $0$ case, the group $W$ therefore has the order (possibly infinite)
\[
2\inf\{\ell\in\mathbb{Z}_+^*: \rho^{2\ell}=1\}.
\]

At that moment, we make a brief detour to connect  \eqref{eq_limunifx} with the  uniformization proposed in \cite[equation 6.5.11]{FIM}, namely
\begin{equation*}
\Check{y}(t) = \frac{y_1+y_4}{2} + \frac{y_4-y_1}{4}\left(t+\frac{1}{t}\right).
 \end{equation*}
Setting
\[
 T(v) = \frac{vz_2-1}{v-z_2} ,
\]
one easily checks the identity $ \widehat{y}(v)=\Check{y}(T(v))$, which by 
 \eqref{eq_limunif3} yields in particular
 \[
 \widehat{y}(1) =  \Check{y}(-1) = y_1 = \widetilde{y}(1/\rho),
 \]
whence, using the properties of the algebraic curve, $\widetilde{x}(1/\rho) = X(y_1)$.
Finally, by \eqref{eq_limunifx}, we obtain that $\rho$ is a root of the second degree equation
\begin{equation}\label{eq_rho}
\rho + \frac{1}{\rho} = 2\frac{x_1+x_4-2x_1x_4+(x_1+x_4-2)X(y_1)}{(x_4-x_1)(1-X(y_1))}.
\end{equation}
The roots of \eqref{eq_rho} are complex conjugate and of modulus one. Choosing then the root satisfying $\arg(\rho)\in[0,\pi]$ and letting
    \begin{equation} \label{eq_lambda}
          \Lambda=\frac{x_1+x_4-2x_1x_4+(x_1+x_4-2)X(y_1)}{2[(X(y_1)-x_1)(X(y_1)-x_4)
          (1-x_1)(x_4-1)]^{1/2}},
     \end{equation}
a direct algebra yields
\[
\arg(\rho)= \frac{\pi}{2}- \arctan(\Lambda).
\]
 
To conclude the proof of Part (I) of  Theorem \ref{main_thm}, it suffices to apply equation 
  \eqref{eq_criter2}, derived in Proposition  \ref{sec_criter2} below, which makes the link between  $\arctan(\Lambda)$ and the ratio $\alpha_3/\alpha_2$.   
All conclusions so far obtained remain true, and even easier to prove, when $C=0$ (see \eqref{eq_C} and the assumption made at the beginning of Section \ref{glulimit}). \bbox 
\subsection{ A second form of the criterion for the zero drift case} 
\label{sec_criter2}
By computing the limit ratio $\alpha_3/\alpha_2$ in a different manner, we shall derive another expression of the angle $\theta$ defined in equation \eqref{exp_Chi}. The following proposition holds.
\begin{prop}
With $\Lambda$ defined in \eqref{eq_lambda} and its equivalent $\widetilde{\Lambda}$ obtained from $\Lambda$ by exchanging $x$ and $y$, we have
\[
\theta= \frac{\pi}{2} -\arctan(\Lambda),
\]
so that $\widetilde{\Lambda}=\Lambda$ and the order of the group $W$ equals   
\[
2\inf\{\ell\in\mathbb{Z}_{+}^*: \ell [1/2-\arctan(\Lambda)/\pi] \in \mathbb{Z}\}.     
\]
\end{prop}
\begin{proof}
A straightforward calculation carried out along the same lines as for the derivation of formulae \eqref{eq_omega0} yields, for any $t<x_1$,
     \begin{multline*} \label{eq_ratio}
          \int_{t}^{x_1} \frac{\text{d}x}{(1-x)[C(x-x_1)(x-x_4)]^{1/2}}= \\
          \frac{1}{[C(x_4-1)(1-x_1)]^{1/2}}
           \Bigg[\frac{\pi}{2}-\arctan\left( \frac{x_1+x_4-2x_1x_4+(x_1+x_4-2)t}{2[(t-x_1)(t-x_4)(x_4-1)(1-x_1)]^{1/2}}\right)\Bigg].
     \end{multline*}
Instantiating now $t=X(y_1)<x_1$ in the last formula, we get exactly
     \begin{equation} \label{eq_criter2}
          \frac{\alpha_3}{\alpha_2}=\frac{1}{2}-\frac{\arctan(\Lambda)}{\pi}. 
     \end{equation}
   \end{proof}

\section{Proof of Part (II) of Theorem \ref{main_thm}}
\label{partIIT}

In this part we prove that in all situations \eqref{eq_cas2}, \eqref{eq_cas3}, \eqref{eq_cas4} and \eqref{eq_cas5} listed in the appendix, the group $W$ is infinite. In fact, it is enough to prove this result for only one case, since two groups corresponding to jump probability sets obtained one another by one of the eight symmetries of the square are necessarily isomorphic, see \cite[Section 2.4]{FIM}. We choose to focus on the particular case \eqref{eq_cas3}. 

Like for the analysis of the case $\drift=0$, we consider the genus 0 case \eqref{eq_cas3} as a continuous limit of the genus 1 case. All the results of Section \ref{properties_kernel} still hold, with $x_3=x_4=\infty$. Accordingly, letting  the parameters $\{p_{i,j}\}_{-1\leq i,j\leq 1}$ be distorted so as to yield  \eqref{eq_cas3}, we get from equations \eqref{eq_omega123} 
     \begin{equation*}
          \left\{\begin{array}{l}
          \omega_1\to i\alpha_1, \ \textrm{with } \alpha_1\in(0,\infty),\\
          \omega_2\to \infty,\\
          \omega_3\to \alpha_3\in(0,\infty).
          \end{array}\right.
     \end{equation*}
In particular, for the same reasons as in Section \ref{limit_uni_10}, the limit group can be interpreted as the group of transformations
     \begin{equation*}
          \langle \omega\mapsto -\omega,\omega\mapsto -\omega+\alpha_3\rangle
     \end{equation*}   
on $\mathbb{C}/(\alpha_1\mathbb{Z})$. This group is obviously infinite, and so is $W$.

The proof of Theorem \ref{main_thm} is terminated. \bbox
\section{Miscellaneous remarks}
For the sake of completeness, we quote hereaftert some facts related to existing works.
     \begin{itemize}
          \item Theorem \ref{main_thm} is quite simple to check, and therefore provides a really  effective criterion. 
          \item  A direct calculation gives $\Delta = -a(1)\widetilde a(1)\sum_{-1\leq i,j\leq 1}i j p_{i,j}$, with $\Delta$ defined in \eqref{def_Delta}. Hence, for the group of order $4$, our criterion clearly agrees with that of \cite{FIM}.
      \item  From Theorem \ref{main_thm}, it becomes clear that the famous \emph{Gessel's walk} (i.e. with jump probabilities  satisfying $p_{1,0}=p_{1,1}=p_{-1,0}=p_{-1,-1}=1/4$) has a group of order $8$. More generally, Theorem \ref{main_thm} leads to another proof of the (non-)finiteness of the group for all combinatorial models (with an underlying genus $0$) appearing in \cite{BK,BMM}.  
      \item In \cite[Section 3]{BMM}, the non-finiteness of the group for some models (including particular instances of case \eqref{eq_cas5}) has been proved via two approaches based on valuation and fixed point arguments, but no general criterion was really obtained. It is also worth noting that there the group was not restricted to the algebraic curve, so that Theorem \ref{main_thm} is, in a sense, more precise. 
          \item As an other straightforward consequence of Theorem \ref{main_thm}, one can check  the random walk considered in \cite{KRSp4}, with 
          \[p_{1,0}=p_{-1,0}=1/2-p_{-1,1}=1/2-p_{1,-1}=\sin^2(\pi/n)/2,
          \]
           has a group of order $2n$, for all $n\geq 3$.
         \item The angle $\theta$ defined in \eqref{exp_Chi}  gives the angle of the cone in which, after a suitable linear transformation, the random walk with transitions $\{p_{i,j}\}_{-1\leq i,j\leq 1}$ has a covariance which equals to some multiple of the identity.
	 \end{itemize}



\appendix

\section{About the genus}
Introduce the \emph{drift} (mean jump vector)
\[
\drift = \Bigg(\sum_{-1\leq i,j \leq 1} ip_{i,j}, \sum_{-1\leq i,j \leq 1} jp_{i,j} \Bigg).
\]

\begin{defn}
\label{def_singular}A random walk is called \emph{singular} if the kernel $K$ defined in \eqref{def_K} is either reducible or of degree $1$ in at least one of the variables $x,y$. 
\end{defn}
The following classification holds, see \cite[Chapter 2]{FIM}.
\begin{lem}\label{lem_genus0}  For all {non-singular} random walks, the Riemann surface corresponding to  \eqref{def_Riemann} has genus $0$ if, and only if, one of the five following  relations takes place:
\begin{eqnarray} 
  \drift &=& \vec{0}, \nonumber\\
   p_{0,1}  &=&  p_{-1,0}  \  = \ \,\,\; p_{-1,1} \ = \  0, \label{eq_cas2}\\
   p_{1,0}  &=&  p_{1,-1}  \ = \  \,\,\; p_{0,-1}  \ = \  0, \label{eq_cas3}\\
   p_{1,0}  &=&  \,\,\;p_{0,1} \ =  \ \,\,\;\,\,\; p_{1,1} \ = \  0,  \label{eq_cas4} \\
   p_{0,-1} &=&  p_{-1,0} \ = \  p_{-1,-1} \ = \ 0.  \label{eq_cas5}
\end{eqnarray}
\end{lem}

\bigskip

\textbf{Acknowledgments.}
K.~Raschel would like to thank I.~Kurkova and J.S.H~van~Leeuwaarden for stimulating discussions. His work was supported by CRC 701, Spectral Structures and Topological Methods in Mathematics at the University of Bielefeld.

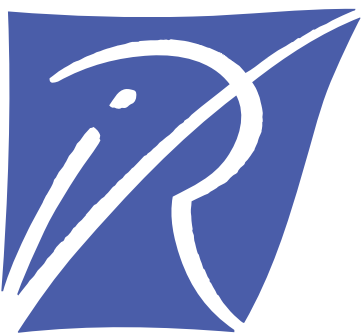
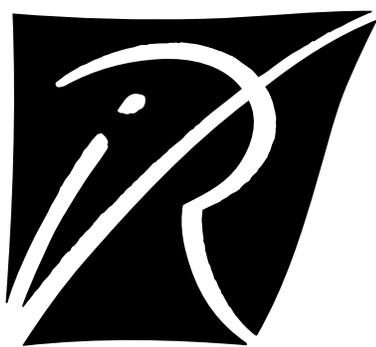
\end{document}

%% file: curve.pstex_t
\begin{picture}(0,0)%
\special{psfile=curve.pstex}%
\end{picture}%
\setlength{\unitlength}{3947sp}%
\begingroup\makeatletter\ifx\SetFigFont\undefined
\def\x#1#2#3#4#5#6#7\relax{\def\x{#1#2#3#4#5#6}}%
\expandafter\x\fmtname xxxxxx\relax \def\y{splain}%
\ifx\x\y   
\gdef\SetFigFont#1#2#3{%
  \ifnum #1<17\tiny\else \ifnum #1<20\small\else
  \ifnum #1<24\normalsize\else \ifnum #1<29\large\else
  \ifnum #1<34\Large\else \ifnum #1<41\LARGE\else
     \huge\fi\fi\fi\fi\fi\fi
  \csname #3\endcsname}%
\else
\gdef\SetFigFont#1#2#3{\begingroup
  \count@#1\relax \ifnum 25<\count@\count@25\fi
  \def\x{\endgroup\@setsize\SetFigFont{#2pt}}%
  \expandafter\x
    \csname \romannumeral\the\count@ pt\expandafter\endcsname
    \csname @\romannumeral\the\count@ pt\endcsname
  \csname #3\endcsname}%
\fi
\fi\endgroup
\begin{picture}(3307,3039)(851,-2458)
\put(2117,-737){\makebox(0,0)[lb]{\smash{\SetFigFont{10}{12.0}{rm}$X(y_1)$}}}
\put(1667,-1052){\makebox(0,0)[lb]{\smash{\SetFigFont{10}{12.0}{rm}$-1$}}}
\put(2881,-2041){\makebox(0,0)[lb]{\smash{\SetFigFont{12}{14.4}{rm}$\M_2$}}}
\put(3422,-782){\makebox(0,0)[lb]{\smash{\SetFigFont{10}{12.0}{rm}$1$}}}
\put(2657,-1007){\makebox(0,0)[lb]{\smash{\SetFigFont{10}{12.0}{rm}$0$}}}
\put(1126,-736){\makebox(0,0)[lb]{\smash{\SetFigFont{10}{12.0}{rm}$X(y_4)$}}}
\put(2566,-376){\makebox(0,0)[lb]{\smash{\SetFigFont{12}{14.4}{rm}$\M_1$}}}
\end{picture}